\documentclass[11pt,a4paper,draft,twoside,notitlepage]{amsart}
\usepackage[utf8]{inputenc}
\usepackage[english]{babel}
\usepackage{amsmath}
\usepackage{amsfonts}
\usepackage{amssymb}
\usepackage[left=2cm,right=2cm,top=2cm,bottom=2cm]{geometry}
\usepackage{hyperref}
\usepackage{xcolor}
\numberwithin{equation}{section}
\usepackage{amsthm}
\newtheorem{thm}{Theorem}[section]
\newtheorem{cor}[thm]{Corollary}
\newtheorem{lem}[thm]{Lemma}

\theoremstyle{definition}

\theoremstyle{remark}

\newcommand{\To}{\longrightarrow}
\newcommand{\g}{\mathfrak{g}}

\newcommand{\Z}{\mathbb{Z}}

\newcommand{\End}{\operatorname{End}}

\newcommand{\gr}[1]{\textit{Gr-}{#1}}
\author{Can Hat\.{i}po\u{g}lu}
\email{hatipoglucan@gmail.com}

\subjclass[2010]{Primary }
\date{\today}
\begin{document}
\title{Injective Hulls of Simple Modules Over Nilpotent Lie Color Algebras}
\maketitle
\begin{abstract}
Using cocycle twists for associative graded algebras, we characterize finite dimensional nilpotent Lie color algebras $L$ graded by arbitrary abelian groups whose enveloping algebras $U(L)$ have the property that the injective hulls of simple right $U(L)$-modules are locally Artinian. We also collect together results on gradings on Lie algebras arising from this characterization.
\end{abstract}
\section{Introduction}
When Jategaonkar proved in \cite{jategaonkar} that fully bounded Noetherian rings satisfy Jacobson's conjecture, a key step in his proof was to observe that over such rings finitely generated essential extensions of simple modules were Artinian. This created an interest in this finiteness property, and a number of Noetherian rings have been studied in relation with it. The obvious question whether this property was shared by every Noetherian ring was answered negatively by Musson in \cite{musson80, musson82}.

Recall that a module $M$ is said to be \emph{locally Artinian} if every finitely generated submodule of it is Artinian. For a Noetherian ring $A$, finitely generated essential extensions of simple right modules are Artinian if and only if the property
\begin{center}
$(\diamond)$ \quad Injective hulls of simple right $A$-modules are locally Artinian 
\end{center}
is satisfied. We should stress that property $(\diamond)$ is sufficient for a Noetherian ring to satisfy Jacobson's conjecture.

Interest in property $(\diamond)$ have increased in recent years, with Noetherian down-up algebras having this property have been characterized in a series of papers by Carvalho \textit{et al.} \cite{carvalho-lomp-pusat}, Carvalho and Musson \cite{carvalho-musson} and Musson \cite{musson12}. Later, finite dimensional solvable Lie superalgebras over algebraically closed fields whose enveloping algebras have property $(\diamond)$ have been characterized in \cite{hatipoglu-lomp} and then a characterization of Ore extensions $k[x][y;\alpha,d]$ of $k[x]$ having property $(\diamond)$ has been obtained in \cite{carvalho-hatipoglu-lomp}. Most recently, a characterization involving torsion theories and some sufficient conditions have been given by the author in \cite{hatipoglu}.

Lie color algebras are natural generalizations of Lie superalgebras. In this note we show using a technique which is sometimes called ``Scheunert's trick'', that the result on nilpotent Lie superalgebras could easily be generalized to nilpotent Lie color algebras. The characterization of nilpotent Lie superalgebras $L$ with property $(\diamond)$ has shown that it only depends on the even part of the Lie superalgebra. We will see that the same holds in this more general setting of Lie color algebras.

Some word on notation. In the rest of the paper, we assume that $k$ is an algebraically closed field of characteristic zero, $G$ is an additive abelian group and all the graded rings and modules are graded by $G$. The group of nonzero elements of $k$ will be denoted by $k^{\ast}$. 
\section{Cocycle Twists of Associative Graded Algebras and Property $(\diamond)$}
Let $A$ be a graded associative algebra over $k$. That is, $A$ is a graded $k$-vector space and there is a family $\{A_{g} \mid g \in G\}$ of subspaces of $A$ such that $A = \oplus_{g\in G}A_{g}$ and $A_{g}A_{h} \subseteq A_{g+h}$ for every $g,h \in G$. The set of all $g \in G$ such that $A_{g} \neq 0$ is called the \textit{support} of the grading. A nonzero element $a$ of $A_{g}$ is said to be a \emph{homogeneous element of degree} $g$, and in this case we will write $|a| = g$ to indicate the degree of $a$. 

A graded right $A$-module is a right $A$-module $M$ such that there exists a family $\{M_{g} \mid g \in G\}$ of $k$-subspaces of $M$ such that $M = \oplus_{g \in G}M_{g}$ and for every $g,h \in G$ we have $M_{g}A_{h} \subseteq M_{g+h}$. For a graded ring $A$, we will denote the category of graded right $A$-modules by $\gr{A}$ and Mod-$A$ will denote the category of right $A$-modules.

By a \emph{2-cocycle} on $G$ we mean a map $\sigma: G \times G \To k^{\ast}$ which satisfies
\begin{equation}\label{2-cocycle}
\sigma(a,b)\sigma(a+b,c) = \sigma(b,c)\sigma(a,b+c)
\end{equation}
for all $a,b,$ and $c \in G$. For our purposes, we will also assume that a 2-cocycle $\sigma$ satisfies the normalization condition $\sigma(0,0)=1$. It follows from this assumption that $\sigma(a,0) = \sigma(0,a) = \sigma(0,0) = 1$.  Note that if $\sigma$ is a 2-cocycle, then the map $\sigma^{-1}:G\times G\To k^{*}$ defined by 
\[\sigma^{-1}(a,b) = \frac{1}{\sigma(a,b)}\]
is also a 2-cocycle.

Let $\sigma$ be a 2-cocycle on $G$. The \emph{cocycle twist} of the algebra $A$ is the algebra $A^{\sigma}$ whose underlying $G$-graded vector space is the same as $A$, and whose multiplication is defined as 
\begin{equation}\label{cocycle-defining-equation}
a\cdot_{\sigma} b = \sigma(|a|,|b|)ab
\end{equation}
for every homogeneous elements $a,b \in A$. Observe that the defining relation \eqref{2-cocycle} of a 2-cocycle is exactly what is needed for the twisted algebra $A^{\sigma}$ to be an associative algebra, because, for homogeneous elements $a, b, c \in A$ we have
\[a\cdot_{\sigma}(b\cdot_{\sigma}c) = a\cdot_{\sigma}(\sigma(|b|,|c|)bc) = \sigma(|a|,|b|+|c|)\sigma(|b|,|c|)abc = \sigma(|a|,|b|)\sigma(|a|+|b|,|c|)abc = (a\cdot_{\sigma}b)\cdot_{\sigma}c\]

There is a corresponding relation between the categories of graded modules over $A$ and its cocycle twist $A^{\sigma}$. Let $M$ be a graded right $A$-module. We define a right $A^{\sigma}$-module structure on $M$ as follows. For homogeneous elements $a \in A^{\sigma}$ and $m \in M$, we define the action of $a$ on $m$ by
\[m*_{\sigma}a = \sigma(|a|,|m|)ma.\]
where the action on the right hand side is the right $A$-action on $M$. We will shortly see that this defines an equivalence of categories between the categories of graded right modules $\gr{A}$ and $\gr{A}^{\sigma}$ but first we need to introduce more terminology. 

For a graded right module $M$ over $A$ and $g \in G$, we define the $g$-\emph{suspension} of $M$ to be the graded right $A$-module $M(g) = \oplus_{h\in G} M(g)_{h}$ where $M(g)_{h} = M_{g+h}$. This defines a functor $T_{g}: \gr{A}\To \gr{A}$  by letting $T_{g}(M) = M(g)$. For $G$-graded rings $A$ and $B$, Gordon and Green \cite{gordon-green} call a functor $F:\gr{A} \To \gr{B}$ a \emph{graded functor} if it commutes (in the appropriate categories) with the $g$-suspension functor for every $g\in G$. A graded functor $U: \gr{A}\To \gr{B}$ is called a \emph{graded equivalence} if there is a graded functor $V:\gr{B}\To \gr{A}$ such that $VU\simeq 1_{\gr{A}}$ and $UV \simeq 1_{\gr{B}}$. If this is the case, the categories $\gr{A}$ and $\gr{B}$ are said to be \emph{graded equivalent}. We are ready to state the following lemma.
\begin{lem}
Let $A$ be a graded ring and $\sigma$ be a 2-cocycle on $G$. Then $\gr{A}$ is graded equivalent to $\gr{A}^{\sigma}$.
\end{lem}
\begin{proof}
Let $(-)^{\sigma}:\gr{A} \To \gr{A}^{\sigma}$ denote the functor which assigns to each graded right $A$-module $M$ the twisted module $M^{\sigma}$, leaving homomorphisms unchanged. Since the functor does not change the underlying $G$-grading on $M$, it is clear that it commutes with the suspension functor $T_{g}$ for every $g \in G$. Moreover, if we let $\sigma^{-1}: G\times G\To k^{*} $ be the 2-cocycle defined by $\sigma^{-1}(a,b)=\frac{1}{\sigma(a,b)}$, then we have $(M^{\sigma})^{\sigma^{-1}} = M$ and hence $(- )^{\sigma^{-1}}\circ (-)^{\sigma} \simeq 1_{\gr{A}}$ and similarly $(- )^{\sigma}\circ (- )^{\sigma^{-1}} \simeq 1_{\gr{A}^{\sigma}}$. So it follows that the categories $\gr{A}$ and $\gr{A}^{\sigma}$ are graded equivalent. 
\end{proof}
We would like to connect $(\diamond)$ property of $A$ to that of $A^{\sigma}$. We already know that there is an equivalence of categories between $\gr{A}$ and $\gr{A^{\sigma}}$, so that the $(\diamond)$ property for graded injective hulls passes from $\gr{A}$ to $\gr{A^{\sigma}}$. The question is whether the same is true for the categories Mod-$A$ and Mod-$A^{\sigma}$. It turns out that this is indeed the case, as the following result shows that graded equivalences between graded module categories give rise to Morita equivalences between module categories. 
\begin{thm}[\cite{menini-nastasescu}, Theorem 1.3. See also \cite{gordon-green} Theorem 5.4.]
Let $A$ and $B$ be $G$-graded rings. Let $\Phi_{A}$ (resp. $\Phi_{B}$) denote the forgetful functor $\gr{A}\To\textrm{Mod-}A$ (resp. $\gr{B}\To\textrm{Mod-}B)$. Then the following statements are equivalent.
\begin{itemize}
\item[(i)] The categories $\gr{A}$ and $\gr{B}$ are graded equivalent;
\item[(ii)] There is a Morita equivalence $L:\textrm{Mod-}A \To \textrm{Mod-}B$ and a graded functor $F:\gr{A}\To \gr{B}$ such that $\Phi_{B}\circ F = L\circ\Phi_{A}$;
\item[(iii)] There exists an object $P\in \gr{A}$ such that $\Phi_{A}(P)$ is a finitely generated projective generator in Mod-$A$ and the graded ring $\End_{A}(P)$ is isomorphic to $B$ as graded rings.
\end{itemize}
\end{thm}
Since $(\diamond)$ is a Morita invariant property, the next result follows from the above theorem.
\begin{cor}\label{the-corollary}
Let $A$ be a $G$-graded associative $k$-algebra and let $\sigma$ be a 2-cocycle on $G$. Then $A$ has property $(\diamond)$ if and only if $A^{\sigma}$ does.
\end{cor}
\section{Cocycle twists of Lie color algebras}
A \emph{commutation factor} $\varepsilon$ on $G$ with values in $k^{\ast}$ is a mapping $\varepsilon: G\times G \To k^{\ast}$ such that
\[\varepsilon(a,b)\varepsilon(b,a) = 1,\quad \varepsilon(a,b+c) = \varepsilon(a,b)\varepsilon(a,c),\quad \varepsilon(a+b,c) = \varepsilon(a,c)\varepsilon(b,c)\]
for all $a,b,c \in G$. It is easy to see from the definition that $\varepsilon(a,a)=\pm 1$ for every $a \in G$ and accordingly we define two sets $G_{+} = \{a \in G \mid \varepsilon(a,a) = 1\}$ and $G_{-} = \{a \in G \mid \varepsilon(a,a) = -1\}$.  Here, $G_{+}$ is a subgroup of $G$ of index at most two and $G_{-}$ is the other coset when the index of $G_{+}$ is two. 

A $G$-graded \emph{Lie color algebra} $L$ with a commutation factor $\varepsilon$ is a $G$-graded vector space $L = \oplus_{g\in G}L_{g}$ with a bracket operation $[,]: L \times L \To L$ which satisfies $[L_{g},L_{h}] \subseteq L_{g+h}$ for all $g,h \in G$, along with color skew symmetry and color Jacobi identities: 

\begin{equation}
[x,y] = -\varepsilon(|x|,|y|)[y,x],
\end{equation}
\begin{equation}
[[x,y],z] = [x,[y,z]] - \varepsilon(|x|,|y|)[y,[x,z]]
\end{equation}
for all homogeneous elements $x,y,z$ of $L$. For example, if $\varepsilon$ is chosen to be the trivial commutation factor defined by $\varepsilon(a,b) = 1$ for all $a,b \in G$, then $L$ is nothing but a $G$-graded ordinary Lie algebra. If $G = \mathbb{Z}_{2}$ and $\varepsilon$ is defined as $\varepsilon(a,b) = (-1)^{ab}$ for all $a,b \in \mathbb{Z}_{2}$, then $L$ is a Lie superalgebra. 

Like the superalgebra case, we have the notion of even and odd elements for Lie color algebras. For a Lie color algebra $L$, we define two sets $L_{+} = \oplus_{g\in G_{+}}L_{g}$ and $L_{-}= \oplus_{g\in G_{-}}L_{g}$. The elements of $L_{+}$ are called \emph{even} while the elements of $L_{-}$ are called the \emph{odd} elements of $L$. 

The \emph{universal enveloping algebra} $U(L)$ of a Lie color algebra $L$ is defined as
\[U(L) = T(L)/J(L)\]
where $T(L)$ is the tensor algebra of $L$ and $J(L)$ is the ideal of $T(L)$ which is generated by the elements $x \otimes y - \varepsilon(|x|,|y|)y\otimes x - [x,y]$ for every homogeneous elements $x,y$ of $L$. The enveloping algebra $U(L)$ of a Lie color algebra is a $G$-graded associative algebra and it is well-known that when $L$ is finite dimensional $U(L)$ is Noetherian. 

Let $L$ be a $G$-graded Lie color algebra with a commutation factor $\varepsilon$. For a 2-cocycle $\sigma$ on $G$, if we define a new multiplication on $L$ by
\begin{equation}\label{cocycle-twist-for-Lie-color-algebras}
[x,y]^{\sigma} = \sigma(|x|,|y|)[x,y]
\end{equation}
for every homogeneous elements $x$ and $y$ of $L$, then with this multiplication $L^{\sigma}$ becomes a $G$-graded Lie color algebra with commutation factor $\varepsilon '= \varepsilon\delta$ where $\varepsilon'(a,b) = \varepsilon(a,b)\delta(a,b)$ and $\delta(a,b) = \sigma(a,b)/\sigma(b,a)$ for all $a,b\in G$. 

The \emph{descending central sequence} of a Lie color algebra $L$ is defined as 
\[\mathcal{C}^{0}(L) = L, \quad \mathcal{C}^{n+1}(L) = [\mathcal{C}^{n}(L), L], \ \forall n \geq 0.\]
$L$ is called \emph{nilpotent} if $\mathcal{C}^{n}(L) = 0$ for some $n$. Before we proceed any further, we note the following easy fact that being nilpotent is preserved under cocycle twists.
\begin{lem}\label{cocycle-twists-are-nilpotent}
Let $L$ be a $G$-graded Lie color algebra and let $\sigma$ be a 2-cocycle on $G$. Then, $L$ is nilpotent if and only if $L^{\sigma}$ is nilpotent.
\end{lem}
\begin{proof}
Obviously, $\mathcal{C}^{0}(L^{\sigma}) = L = \mathcal{C}(L)$. It is also clear that as $k$-spaces $\mathcal{C}^{i+1}(L^{\sigma}) = \mathcal{C}^{i+1}(L)$ for any $i \geq 0$. Hence, $\mathcal{C}^{n}(L) = 0$ if and only if $\mathcal{C}^{n}(L^{\sigma}) = 0$, \textit{i.e.} $L$ is nilpotent if and only if $L^{\sigma}$ is nilpotent.
\end{proof}
Cocycle twists can be applied to pass from a Lie color algebra to a Lie superalgebra and we briefly explain this process now. Let $L$ be a $G$-graded Lie color algebra with a commutation factor $\varepsilon$. If $\sigma$ is a 2-cocycle on $G$, we know that $L^{\sigma}$ is a $G$-graded Lie color algebra with a commutation factor $\varepsilon' = \varepsilon\delta$, where $\delta$ is as in the lemma. Let $\varepsilon_{0}: G \times G\To k^{*}$ denote the superalgebra commutation factor, which is defined by the rule 
\[\varepsilon_{0}(a,b) = 1\]
if $a\in G_{+}$ or $b\in G_{+}$ and 
\[\varepsilon_{0}(a,b) = -1\]
when both $a,b \in G_{-}$. The following result is the crucial step in the process of passing from a Lie color algebra to a Lie superalgebra.
\begin{thm}\cite[Theorem 2.3]{bahturin-montgomery}
Let $G$ be an arbitrary abelian group, and $k$ an arbitrary commutative ring with $1$ and with group of units $k^{*}$. Then for any commutation factor $\varepsilon:G \times G \To k^{*}$ there exists a 2-cocycle $\sigma:G \times G \To k^{*}$ such that if we set $\delta(g,h) = \sigma(g,h)/\sigma(h,g)$, then $\varepsilon\delta = \varepsilon_{0}$.
\end{thm}
That is, for any $G$-graded Lie color algebra $L$ with a commutation factor $\varepsilon$, we can find a 2-cocycle $\sigma$ such that the twisted Lie color algebra $L^{\sigma}$ is a Lie superalgebra $L^{\sigma} = L_{0}^{\sigma} \oplus L^{\sigma}_{1}$ where $L_{0}^{\sigma} = L_{+}$ and $L_{1}^{\sigma} = L_{-}$. 

Hence, the discussions of the previous section apply to Lie color algebras as well and there is a corresponding equivalence of categories between the graded representations of a Lie color algebra $L$ and the graded representations of its cocycle twist $L^{\sigma}$. While this correspondence only exists between the graded modules over Lie color algebras and Lie superalgebras, we are interested in modules over the enveloping algebras of such algebras. Fortunately, the cocycle twists of Lie color algebras are connected to their enveloping algebras in the following way.
\begin{lem}\cite[Theorem 2]{scheunert}\label{isomorphism}
Let $L$ be a $G$-graded Lie color algebra and let $\sigma$ be a 2-cocycle on $G$. Then there is an algebra isomorphism
\[U(L^{\sigma}) \cong U(L)^{\sigma}\]
where $U(L)$ is considered as a $G$-graded associative algebra and $U(L)^{\sigma}$ is its cocycle twist.
\end{lem}
Together with Corollary~\ref{the-corollary}, we have:
\begin{lem}\label{the-lemma}
Let $L$ be a $G$-graded Lie color algebra and let $\sigma$ be a 2-cocycle on $G$. Then $U(L)$ has property $(\diamond)$ if and only if $U(L^{\sigma})$ does.
\end{lem}
The previous lemma and the Morita equivalence between $U(L)$ and its cocycle twists allow us to reduce our study to that of the enveloping algebra of a Lie superalgebra. Recall that a \emph{central abelian direct factor} of a Lie algebra $L$ is an abelian Lie subalgebra $A$ of $L$ such that $L = A \times B$ for some Lie subalgebra $B$ of $L$. We can now state the main result of this note.
\begin{thm}\label{diamond-property-for-color}
Let $L$ be a finite dimensional nilpotent $G$-graded Lie color algebra with a commutation factor $\varepsilon$. Let $\sigma$ be the 2-cocycle such that the twisted algebra $L^{\sigma}$ is a $G$-graded Lie superalgebra with $L_{0}^{\sigma} = L_{+}$ and $L_{1}^{\sigma} = L_{-}$. The following statements are equivalent. 
\begin{itemize}
\item[(i)] $U(L)$-Mod has property $(\diamond)$;
\item[(ii)] $U(L^{\sigma})$-Mod has property $(\diamond)$;
\item[(iii)] $U(L_{0}^{\sigma})$-Mod has property $(\diamond)$;
\item[(iv)] Up to a central abelian direct factor, $L_{0}^{\sigma}$ is isomorphic to either
\begin{itemize}
	\item[(a)] a nilpotent Lie algebra with an abelian ideal of codimension one,
	\item[(b)] the five dimensional Lie algebra $L_{5}$  with basis $\{e_{1},\ldots, e_{5}\}$ and nonzero brackets $[e_{1},e_{2}] = e_{3},\ [e_{1},e_{3}]=e_{4},\ [e_{2},e_{3}]=e_{5}$,
	\item[(c)] the six dimensional Lie algebra $L_{6}$ with basis $\{e_{1}, \ldots, e_{6}\}$ and nonzero brackets $[e_{1},e_{3}]=e_{4},\ [e_{2},e_{3}]=e_{5}, \ [e_{1},e_{2}]=e_{6}$. 
\end{itemize}
\end{itemize}
\end{thm}
\begin{proof}
The equivalence of $(i)$ and $(ii)$ follows from Lemma~\ref{the-lemma}. $(ii) \Leftrightarrow (iii) \Leftrightarrow (iv)$ from Lemma~\ref{cocycle-twists-are-nilpotent} and \cite[Theorem 1.1]{hatipoglu-lomp}. 
\end{proof}
\section{Nilpotent Lie algebras of almost maximal index}
A Lie algebra $\g$ is said to have \textit{almost maximal index} if $\textrm{ind}(\g) = \dim(\g) - 2$. According to \cite[Proposition 5.3]{hatipoglu-lomp}, a finite dimensional nilpotent Lie algebra has almost maximal index if and only if it is one of the Lie algebras appearing in Theorem~\ref{diamond-property-for-color}(iv).

Finite dimensional Lie algebras with an abelian ideal of codimension one are in one-to-one correspondence with finite dimensional vector spaces $V$ and nilpotent endomorphisms $f:V \To V$. Given such a vector space and a nilpotent endomorphism, one defines a Lie algebra structure on the vector space $L = kx \oplus V$ by defining $[x,v] = f(v)$. A typical example of this construction is the $n$-dimensional standard filiform Lie algebra, which is the Lie algebra $L_{n}$ with a basis $\{e_{1}, e_{2}, \ldots, e_{n}\}$ and whose nonzero brackets are $[e_{1}, e_{i}] = e_{i+1}$ for $i = 2, 3, \ldots, n-1$. For instance the three dimensional Heisenberg Lie algebra is the three dimensional standard filiform Lie algebra. 

The $n$-dimensional standard filiform Lie algebra $L_{n}$ is realized in the above sense as a vector space $L_{n} = kx_{1} \oplus V$ where $V$ is the span of the elements $x_{2}, x_{2}, \ldots, x_{n}$ with a nilpotent endomorphism $f: V \To V$ given by $f(x_{i}) = x_{i+1}$ for all $i = 2, 3, \ldots, n-1$. 

It turns out that the standard filiform Lie algebra is the only finite dimensional nilpotent Lie algebra with an abelian ideal of codimension one, up to isomorphism. First note that for any nilpotent endomorphism $f$ on a finite dimensional vector space $V$, there exists a basis $\{v_{1}, v_{2}, \ldots, v_{n}\}$ of $V$ such that $f(v_{i})$ is either zero or equal to $v_{i+1}$ for all $i=1,2,\ldots, n-1$. Now, let $L = kx \oplus V$ be a finite dimensional nilpotent Lie algebra with an abelian ideal $V$. Then, $V$ has a basis $\{v_{1}, v_{2}, \ldots, v_{n}\}$ such that $[x,v_{i}]$ is either zero or $v_{i+1}$, for all $i = 1,2,\ldots, n-1$. Without loss of generality, we may assume that $[x,v_{1}] \neq 0$, applying a change of basis if necessary. If $[x, v_{i}]$ is nonzero for each $i$, then $L$ is standard filiform. If $[x,v_{i}] = 0$ for some $i > 1$, then $L$ can be written as the direct sum of a standard filiform Lie algebra $L_{1} = \langle x, v_{1}, v_{2}, \ldots,v_{i}\rangle$ and an abelian Lie ideal $L_{2} = \langle v_{i+1}, v_{i+2}, \ldots, v_{n}\rangle$, therefore proving our claim.

The above discussion shows that if $L$ is a finite dimensional nilpotent Lie algebra, then $U(L)$ satisfies property $(\diamond)$ if and only if $L$ is isomorphic, up to a central abelian direct factor, to $L_{5}, L_{6}$ or to a standard filiform Lie algebra. 

\section{Gradings on nilpotent Lie algebras of almost maximal index}
In this section we collect together results on the classification of group gradings on finite dimensional nilpotent Lie algebras of almost maximal index. Two gradings $\Gamma: L = \oplus L_{g}$ and $\Gamma ': \oplus_{g}L'_{g}$ on a Lie algebra $L$ by an abelian group $G$ are called \textit{isomorphic} if there is a Lie algebra automorphism $\varphi: L \To L$ such that $\varphi(L_{g}) = L_{g}'$. 

\subsection{Gradings on $L_{5}$ and $L_{6}$}

Let us first consider the gradings on $L_{5}$ and $L_{6}$. These are free nilpotent Lie algebras, in particular, $L_{5}$ is the free Lie algebra with two generators of step three and $L_{6}$ is the free nilpotent Lie algebra with three generators of step two. Finite dimensional free nilpotent Lie algebras are members of the larger class of \textit{free algebras of finite rank in nilpotent varieties of algebras}. The gradings on these algebras have been classified by Bahturin in \cite{bahturin}. If $F_{n}$ is such an algebra with a generating set $\{x_{1}, x_{2}, \ldots, x_{n}\}$ then there is a standard $\Z^{n}$-grading on $F_{n}$ obtained in the following way. We let $\alpha = (d_{1}, d_{2}, \ldots, d_{n}) \in \Z^{n}$ and $F_{n}^{\alpha}$ be the span of all the monomials whose degree with respect to each generator $x_{i}$ is $d_{i}$, $i = 1, 2, \ldots, n$. Then the subspaces $F_{n}^{\alpha}$ form a $\Z^{n}$-grading on $F_{n}$ as $F_{n} = \displaystyle\oplus_{\alpha \in \Z^{n}} F_{n}^{\alpha}$ (see \cite{bahturin}).

Then the corresponding standard gradings on $L_{5}$ and $L_{6}$ are as follows. 
\begin{enumerate}
\item
If $L = L_{5}$ then $L = \oplus_{(a,b) \in\Z^{2}}L_{(a,b)}$ where $L_{(0,1)} = \langle e_{1}\rangle,L_{(0,1)} = \langle e_{2}\rangle, L_{(1,1)} = \langle e_{3}\rangle,\ L_{(2,1)} = \langle e_{4} \rangle,\ L_{(1,2)} = \langle e_{5} \rangle$ and all other summands are zero. 
\item If $L = L_{6}$, then $L = \oplus_{(a,b,c) \in Z^{3}}$ where $L_{(1,0,0)} = \langle e_{1}\rangle$, $L_{(0,1,0)} = \langle e_{2}\rangle,\ L_{(0,0,1)} = \langle e_{3}\rangle,\ L_{(1,0,1)} = \langle e_{4}\rangle,\ L_{(0,1,1)} = \langle e_{5}\rangle$, and $L_{(1,1,0)} = \langle e_{6} \rangle$ and all the other summands are zero. 
\end{enumerate}

Essentially, the above standard gradings are the only gradings on these Lie algebras in the sense that any grading of these algebras are induced from the standard gradings. More precisely, if $G$ and $H$ are abelian groups and $\alpha:G \To H$ is a group homomorphism, then a grading $\Gamma': V = \oplus_{h \in H} V'_{h}$ is said to be induced from the grading $\Gamma: V = \oplus_{g \in G}V_{g}$ if 

\[V'_{h\in H} = \bigoplus_{g\in G:\ \alpha(g) = h}V_{g}.\]
Then, it follows from \cite[Theorem 1]{bahturin} that any $G$-grading $\Gamma$ on $L_{n}$ is induced from the standard $\Z^{n}$-grading by a homomorphism $\alpha: \Z^{n} \To H$, $n = 2,3$.

For two gradings $\Gamma: \oplus_{g\in G}A_{g}$ and $\Gamma': \oplus_{h \in H}A'_{h}$ we say that $\Gamma'$ is a \textit{refinement} of $\Gamma$ (or that $\Gamma$ is a \textit{coarsening} of $\Gamma'$) if for each $h\in H$ there exists $g\in G$ such that $A'_{h} \subset A_{g}$. A grading $\Gamma$ which does not have a proper refinement is called \textit{fine}. It turns out that the standard grading is also the only fine abelian group grading of the Lie algebras $L_{5}$ and $L_{6}$, up to equivalence, see \cite[Corollary 3]{bahturin}. 

\subsection{Gradings on Standard Filiform Lie Algebras}

Gradings on standard filiform Lie algebras have been classified by Bahturin et al. in \cite{bahturin-et-al} and they are parallel to gradings on free nilpotent Lie algebras. Let $L = L_{n}$ denote the $n$-dimensional standard filiform Lie algebra. Then $L$ has the standard grading defined by $L = \oplus_{(a,b)\in \Z^{2}}L_{(a,b)}$ where $L_{(1,0)} = \langle e_{1} \rangle$, $L_{(s-2,1)} = \langle e_{s}\rangle$ for all $s = 2,3, \ldots, n$ and all other summands are zero. In this case, any $G$-grading of $L_{n}$ is isomorphic to a coarsening of the standard $\Z^{n}$-grading \cite[Theorem 12]{bahturin-et-al}.

\bigskip
\footnotesize
\textsc{Research Center for Theoretical Sciences, Mathematics Division, National Cheng Kung University, Tainan, Taiwan}.\par\nopagebreak
Currently at \textsc{the American University of the Middle East, Department of Mathematics and Statistics, Kuwait}

\begin{thebibliography}{99}
\bibitem{bahturin}
Bahturin, Yuri. 
Group Gradings on Free Algebras of Nilpotent Varieties of Algebras.
\emph{Serdica Math. J.}, 38, (2012), 1 - 6. 

\bibitem{bahturin-montgomery}
Bahturin, Yuri; Montgomery, Susan.
PI-envelopes of Lie superalgebras. 
\emph{Proc. Amer. Math. Soc}. 127 (1999), no. 10, 2829-2839. 

\bibitem{bahturin-et-al}
Bahturin, Yuri; Goze, Michel; Remm, Elisabeth.
Group Gradings on Filiform Lie Algebras.
\emph{Comm. Algebra}, 44: 40 - 62, 2016.

\bibitem{carvalho-hatipoglu-lomp}
Carvalho, Paula A. A. B.; Hatipoğlu, Can; Lomp, Christian; 
Injective Hulls of Simple Modules over Differential Operator Rings. 
\emph{Comm. Algebra} 43 (2015), no. 10, 4221-4230. 

\bibitem{carvalho-lomp-pusat}
Carvalho, Paula A. A. B.; Lomp, Christian; Pusat-Yilmaz, Dilek.
Injective modules over down-up algebras. 
\emph{Glasg. Math. J}. 52 (2010), no. A, 53-59. 

\bibitem{carvalho-musson}
Carvalho, Paula A. A. B.; Musson, Ian M.
Monolithic modules over Noetherian rings. 
\emph{Glasg. Math. J}. 53 (2011), no. 3, 683-692. 

%\bibitem{chen-et-al}
%Chen, X.-W.; Silvestrov, S. D.; Van Oystaeyen, F.
%Representations and cocycle twists of color Lie algebras.
%\emph{Algebr. Represent. Theory} 9 (2006), no. 6, 633-650. 

\bibitem{gordon-green}
Gordon, Robert; Green, Edward L.
Graded Artin algebras. 
\emph{J. Algebra} 76 (1982), no. 1, 111-137.

\bibitem{hatipoglu}
Hatipoglu, C. 
Stable Torsion Theories and the Injective Hulls of Simple Modules. 
\emph{Int. Electron. J. Algebra}, 16, 2014, 89-98. 

\bibitem{hatipoglu-lomp}
Hatipo\u{g}lu, Can; Lomp, Christian.
Injective hulls of simple modules over finite dimensional nilpotent complex Lie superalgebras. 
\emph{J. Algebra} 361 (2012), 79-91. 

\bibitem{jategaonkar}
Jategaonkar, Arun Vinayak.
Jacobson's conjecture and modules over fully bounded Noetherian rings. 
\emph{J. Algebra} 30 (1974), 103-121.

\bibitem{menini-nastasescu}
Menini, C.; Năstăsescu, C.
When is R-gr equivalent to the category of modules? 
\emph{J. Pure Appl. Algebra} 51 (1988), no. 3, 277-291. 

\bibitem{musson80}
Musson, I. M. 
Injective modules for group rings of polycyclic groups. I, II. 
\emph{Quart. J. Math. Oxford Ser}. (2) 31 (1980), no. 124, 429-448, 449-466.

\bibitem{musson82}
Musson, I. M.
Some examples of modules over Noetherian rings. 
\emph{Glasgow Math. J}. 23 (1982), no. 1, 9-13. 

\bibitem{musson12}
Musson, Ian M.
Finitely generated, non-Artinian monolithic modules.
\emph{New trends in noncommutative algebra}, 211-220, 
Contemp. Math., 562, Amer. Math. Soc., Providence, RI, 2012.

\bibitem{scheunert}
Scheunert, M.
Generalized Lie algebras.
\emph{J. Math. Phys}. 20 (1979), no. 4, 712-720. 
\end{thebibliography}
\end{document}